\documentclass[12pt,a4paper,oneside]{amsart}
\usepackage{color,amssymb,latexsym,amsfonts,textcomp,lastpage,fancyhdr,calc,graphicx}
\usepackage{amsmath,amstext,amsthm,amssymb,amsxtra}
\usepackage{txfonts}
\usepackage[left=2.5cm,right=2.5cm,bottom=2cm,top=2cm]{geometry} 

\newtheorem{theorem}{Theorem}
\newtheorem{corollary}[theorem]{Corollary}
\newtheorem{lemma}{Lemma}

\newtheorem{definition}{Definition}[section]
\theoremstyle{definition}

\newcommand{\beql}[1]{\begin{equation}\label{#1}}
\newcommand{\eeq}{\end{equation}}
\newcommand{\comment}[1]{}

\newcommand{\Abs}[1]{{\left|{#1}\right|}}

\newcommand{\Set}[1]{{\left\{{#1}\right\}}}

\newcommand{\RR}{{\mathbb R}}

\newcommand{\ZZ}{{\mathbb Z}}

\newcommand{\inner}[2]{{\langle #1, #2 \rangle}}

\newcommand{\supp}{{\rm supp\,}}

\newcommand{\ft}[1]{\widehat{#1}}

\newcounter{rem}
\newcounter{step}
\newcounter{mysec}
\newcounter{mysubsec}[mysec]
\newcounter{othm}
\def\theothm{\Alph{othm}}
\newenvironment{othm}{
  \em
  \vskip 0.10in
  \refstepcounter{othm}
  \noindent{\bf Theorem\ \theothm}
}

\begin{document}

\title{Periodicity of the spectrum in dimension one}

\author[A. Iosevich]{{Alex Iosevich}}
\address{A.I.: Department of Mathematics, 915 Hylan Building, University of Rochester,
Rochester, NY 14627, U.S.A.}
\email{iosevich@math.rochester.edu}
\thanks{A.I.: Supported by NSF grant DMS10-45404}

\author[M. Kolountzakis]{{Mihail N. Kolountzakis}}
\address{M.K.: Department of Mathematics, University of Crete, Knossos Ave., GR-714 09, Iraklio, Greece}
\email{kolount@math.uoc.gr}
\thanks{M.K.: Supported by research grant No 3223 from the Univ.\ of Crete and by
NSF grant DMS10-45404 and grants of the University of Rochester, whose hospitality is gratefully
acknowledged}

\date{\today}

\begin{abstract}
A bounded measurable set $\Omega$, of Lebesgue measure 1,
in the real line is called spectral if there is a set $\Lambda$ of real numbers (``frequencies'')
such that the exponential functions $e_\lambda(x) = \exp(2\pi i \lambda x)$, $\lambda\in\Lambda$, form
a complete orthonormal system of $L^2(\Omega)$. Such a set $\Lambda$ is called a {\em spectrum} of $\Omega$.
In this note we prove that any spectrum $\Lambda$ of a bounded measurable
set $\Omega\subseteq\RR$ must be periodic.
\end{abstract}

\maketitle


\ 

\noindent{\bf AMS Primary Classification:} 42B99

\tableofcontents 

\section{Tilings, spectral sets and periodicity}

\subsection{Spectra of domains in Euclidean space and the Fuglede Conjecture}
Let $\Omega \subseteq \RR^d$ be a bounded measurable set
and let us assume for simplicity that $\Omega$ has Lebesgue measure 1.
The concept of a spectrum of $\Omega$ that we deal with in this paper may be interpreted as
a way of using Fourier series for functions defined on $\Omega$ with non-standard frequencies.
It was introduced by Fuglede \cite{fuglede1974operators} 
who was studying a problem of Segal on the extendability of the partial differential operators
(on $C_c(\Omega)$)
$$
\frac{\partial}{\partial x_1}, \frac{\partial}{\partial x_2}, \ldots, \frac{\partial}{\partial x_d}
$$
to commuting operators on all of $L^2(\Omega)$.

\begin{definition}
A set $\Lambda \subseteq \RR^d$
is called a {\em spectrum} of $\Omega$ (and $\Omega$ is said to be a {\em spectral set})
if the set of exponentials
$$
E(\Lambda) = \Set{e_\lambda(x)=e^{2\pi i \lambda\cdot x}:\ \lambda\in\Lambda}
$$
is a complete orthonormal set in $L^2(\Omega)$.
\end{definition}
(The inner product in $L^2(\Omega)$ is $\inner{f}{g} = \int_\Omega f \overline{g}$.)

It is easy to see (see, for instance, \cite{kolountzakis2004milano}) that the orthogonality of $E(\Lambda)$
is equivalent to the {\em packing condition}
\beql{packing-condition}
\sum_{\lambda\in\Lambda}\Abs{\ft{\chi_\Omega}}^2(x-\lambda) \le 1,\ \ \mbox{a.e. ($x$)},
\eeq
as well as to the condition
\beql{zeros-condition}
\Lambda-\Lambda \subseteq \Set{0} \cup \Set{\ft{\chi_\Omega}=0}.
\eeq
The orthogonality and completeness of $E(\Lambda)$ is in turn equivalent to the {\em tiling condition}
\beql{tiling-condition}
\sum_{\lambda\in\Lambda}\Abs{\ft{\chi_\Omega}}^2(x-\lambda) = 1,\ \ \mbox{a.e. ($x$)}.
\eeq
These equivalent conditions follow from the identity
\beql{inner}
\inner{e_\lambda}{e_\mu} = \int_\Omega e_\lambda \overline{e_\mu} = \ft{\chi_\Omega}(\mu-\lambda)
\eeq
and from the completeness of all the exponentials in $L^2(\Omega)$.
Condition \eqref{packing-condition} is roughly expressing the validity of Bessel's inequality for the
system of exponentials $E(\Lambda)$ while condition \eqref{tiling-condition} says that Bessel's inequality
holds as equality.

If $\Lambda$ is a spectrum of $\Omega$ then so is any translate of $\Lambda$ but there may be other spectra as well.

{\em Example:} If $Q_d = (-1/2, 1/2)^d$ is the cube of
unit volume in $\RR^d$ then
$\ZZ^d$ is a spectrum of $Q_d$.
Let us remark here that
there are spectra of $Q_d$ which are very different from translates of the lattice $\ZZ^d$
\cite{iosevich1998spectral,lagarias2000orthonormal,kolountzakis2000packing}.

In the one dimensional case, which will concern us in this paper,
condition \eqref{zeros-condition}
implies that the set $\Lambda$ has gaps bounded below by a positive number, the smallest
positive zero of $\ft{\chi_\Omega}$.
(Note that, since $\Omega$ is a bounded set, the function $\ft{\chi_\Omega}(\xi)$ can be defined for
all complex $\xi$ and is an entire function. This guarantees that its zeros are a discrete set.)

\noindent{\bf The Fuglede or Spectral Set Conjecture.}
Research on spectral sets has been driven for many years by a conjecture of Fuglede
\cite{fuglede1974operators} which stated that a set $\Omega$ is spectral if and only
if it is a translational tile. A set $\Omega$ is a translational tile if
we can translate copies of $\Omega$ around and fill space without overlaps.
More precisely there exists a set $S \subseteq \RR^d$ such that
\beql{tiling}
\sum_{s\in S} \chi_\Omega(x-s) = 1,\ \ \mbox{a.e. ($x$)}.
\eeq

One can extend the definition of translational tiling to functions from sets.
\begin{definition}
We say that the nonnegative
function $f:\RR^d\to\RR$ tiles by translation with the set $S \subseteq \RR^d$
if
$$
\sum_{s\in S} f(x-s) = \ell,\ \mbox{for almost every $x\in\RR^d$},
$$
where $\ell$ is a constant (the {\em level} of the tiling).
\end{definition}

Thus the question of spectrality for a set $\Omega$ is essentially a tiling question
for the function $\Abs{\ft{\chi_\Omega}}^2$ (the {\em power-spectrum}).
Taking into account the equivalent condition \eqref{tiling-condition}
one can now, more elegantly, restate the Fuglede Conjecture as the equivalence
\beql{fuglede-conjecture}
\chi_\Omega \mbox{ tiles $\RR^d$ by translation at level 1} \Longleftrightarrow
\Abs{\ft{\chi_\Omega}}^2 \mbox{ tiles $\RR^d$ by translation at level 1}.
\eeq
In this form the conjectured equivalence is perhaps more justified.
However this conjecture is now known to be false in both directions if $d\ge 3$
\cite{tao2004fuglede,matolcsi2005fuglede4dim,kolountzakis2006hadamard,kolountzakis2006tiles,farkas2006onfuglede,farkas2006tiles},
but remains open in dimensions $1$ and $2$ and it is not out of the question
that the conjecture is true if one restricts the domain $\Omega$ to be convex. (It is known
that the direction ``tiling $\Rightarrow$ spectrality'' is true in the case of convex domains;
see for instance \cite{kolountzakis2004milano}.)
The equivalence \eqref{fuglede-conjecture} is also known, from the time of Fuglede's
paper \cite{fuglede1974operators}, to be true if one adds the word {\em lattice} to both sides
(that is, lattice tiles are the same as sets with a lattice spectrum).

\subsection{Periodicity of spectra and tilings}

The property of periodicity is a very important property for a tiling.
\begin{definition}
A set $S \subseteq \RR^d$ is called (fully) periodic
if there exists a lattice $L \subseteq \RR^d$ (a discrete subgroup of $\RR^d$ with $d$
linearly independent generators; the period lattice) such that
$S+t = S$ for all $t \in L$.
We call a translation tiling periodic if the set of translations is periodic.
\end{definition}
As an example of the importance of periodicity for a tiling we mention its connection
to decidability \cite{robinson1971undecidability}, a question to which the study of tilings
has provided several examples and problems.
Although the general problem of tiling (not restricting the motions to be translations) is
undecidable, it is not hard to see that, when the assumption of periodicity is added,
the problem becomes decidable. 
Roughly, if one knows a priori that a set $\Omega$ admits periodic tilings, if it admits any,
then the question ``Does $\Omega$ admit a tiling?'' can be answered algorithmically by
simultaneously enumerating all possible counterexamples to tiling (if a tiling does not
exist then the obstacle will show up at some finite stage) as well as all possible tilings of finite
regions.
If a tiling does not exist then the first enumeration will produce a counterexample.
Otherwise, if a tiling exists then a periodic tiling exists and
one of the finite regions that can be tiled will show this
periodicity and can therefore be extended to all space. More details of this argument
can be found at \cite{robinson1971undecidability}.

The so-called {\em Periodic Tiling Conjecture} \cite{grunbaum1986tilings,lagarias1997spectral}
should be mentioned: if a set $\Omega$ tiles $\RR^d$ by translations (at level 1) then it can
also tile $\RR^d$ by a periodic set of translations.
By the argument sketched above this conjecture implies decidability for whatever class of sets
$\Omega$ it holds true.
For instance, when one considers finite discrete sets $\Omega \subseteq \ZZ^d$ (tiling is defined
analogously to the continuous case) then one could prove the decidability of the tiling
question if one managed to first prove that any such set $\Omega$ that tiles $\ZZ^d$ by translation
can also tile by a periodic set of translations.
Both questions are open for $d\ge 2$
(but see \cite{szegedy1998algorithms,wijshoff1984arbitrary} for some special cases).
For $d=1$ all translational tilings by finite subsets of $\ZZ$ are necessarily periodic
\cite{newman1977tesselations} and the problem is decidable.
Another class of tilings where the Periodic Tiling Conjecture holds is the case when $\Omega$ is assumed
to be a convex polytope in $\RR^d$, for any $d$ \cite{venkov1954class,mcmullen1980convex}.

In dimension $d=1$ it is known
\cite{leptin1991uniform,lagarias1996tiling,kolountzakis1996structure}
that all translational tilings by a bounded measurable set
are necessarily periodic.
More generally it is known that whenever $f \ge 0$ is an integrable function on the real line
which tiles the real line by translation with a set of translates $S$ then
$S$ is of the form
\beql{quasi-periodic}
S = \bigcup_{j=1}^J (\alpha_j \ZZ + \beta_j),
\eeq
where the real numbers $\alpha_j$ are necessarily commensurable (and $S$ is in that case
periodic) if the tiling is indecomposable (cannot be made up by superimposing other tilings).
But this result is not applicable to the periodicity of spectra as the power-spectrum
$\Abs{\ft{\chi_\Omega}}^2$ is never of compact support when $\Omega$ is bounded
(a qualitative expression of the {\em uncertainty principle}).

The question of periodicity of one-dimensional spectra was explicitly raised in \cite{laba2002spectral}.
It was recently proved (first in \cite{bose2010spectrum} and then a simplified
proof was given in \cite{kolountzakis2011periodic}) that if $\Omega$
is a finite union of intervals in the real line then any spectrum of $\Omega$ is periodic.
See also \cite{lagarias1997spectral} where periodicity of spectra and of tilings plays
an important role.

\begin{othm}\label{th:bose-madan}
{\rm [Bose and Madan \cite{bose2010spectrum}; Kolountzakis \cite{kolountzakis2011periodic}]}\\
If $\Omega = \bigcup_{j=1}^n (a_j, b_j) \subseteq \RR$ is a finite union of intervals of total
length $1$ and $\Lambda \subseteq \RR$ is a spectrum of $\Omega$ then there exists a positive integer
$T$ such that $\Lambda+T=\Lambda$.
\end{othm}

\ 

Our purpose in this note is to improve Theorem \ref{th:bose-madan}
by removing the assumption that $\Omega$ is a finite union of intervals.
\begin{theorem}\label{th:main}
Suppose that $\Lambda$ is a spectrum of $\Omega \subseteq \RR$, a bounded measurable set of
measure 1. Then $\Lambda$ is periodic and any period is a positive integer.
\end{theorem}

The proof of Theorem \ref{th:main} is given in \S\ref{sec:proof}.

\begin{corollary}\label{cor:multiple-tiling}
If $\Omega$, a bounded measurable set of
measure 1, is spectral then $\Omega$ tiles the real line at some integer
level $T$ when translated at the locations $T^{-1} \ZZ$.
\end{corollary}

\begin{proof}
Let $\Lambda$ is a spectrum of $\Omega$.
By Theorem \ref{th:main} we know that $\Lambda$ is a periodic set and let $T$ be one of its periods:
$\Lambda+T=\Lambda$.
Then we have $\Lambda = T\ZZ + \Set{\ell_1,\ldots,\ell_T}$
(the number of elements in each period must be $T$ in order for $\Lambda$ to have
density 1, hence $T$ is an integer), and, by \eqref{zeros-condition},
this implies that $\ft{\chi_\Omega}(nT) = 0$ for all nonzero $n \in \ZZ$.
Hence $\Omega$ tiles $\RR$ when translated at $T^{-1}\ZZ$ (see, e.g.\ \cite{kolountzakis2004milano})
at level $T$.
\end{proof}

Theorem \ref{th:main} is not true in dimension higher than 1. For instance,
even when $\Omega$ is as simple as a cube, it may have spectra that are not periodic
\cite{lagarias2000orthonormal,iosevich1998spectral,kolountzakis2000packing}.

\section{Proof of periodicity for spectra in dimension 1}
\label{sec:proof}

\subsection{The spectrum as a double sequence of symbols}
\label{sec:symbols}

Because of \eqref{zeros-condition} we have that the gap between any two elements
of $\Lambda$ is bounded below by $\delta>0$, the smallest positive zero of $\ft{\chi_\Omega}$.
Let us now observe
that the gap between successive elements of $\Lambda$ is also bounded above
by a constant that depends only on $\Omega$.

\begin{lemma}\label{lm:bounded-gaps}
If $\Omega\subseteq\RR$ is a bounded measurable set of measure 1
then there is a finite number $\Delta>0$
such that if $\Lambda$ is any spectrum of $\Omega$
then the gap between any two successive
elements of $\Lambda$ is at most $\Delta$.
\end{lemma}

\begin{proof}
Lemma \ref{lm:bounded-gaps} is essentially a special case of Lemma 2.3 of 
\cite{kolountzakis1996structure}.
In that Lemma it is proved that if $0 \le f \in L^1(\RR)$ tiles the line with a set $A$,
$$
\sum_{a \in A} f(x-a) = w,\ \ \mbox{for almost all $x\in\RR$, with $w>0$ a constant},
$$
then the set $A$ has asymptotic density equal to $\rho = \frac{w}{\int f}$.
This means that the ratio
$$
\frac{\Abs{A \cap I}}{\Abs{I}}
$$
tends to $\rho$ as the length of the interval $I$ tends to infinity.
The convergence is {\em uniform} over the choice of the set $A$ and the location of the interval $I$.
\footnote{
Inequality (2.4) speaks of $N_A(T) = \Abs{A \cap [-T,T]}$,
but none of the other quantities that appear in it depend on $A$.
This means that inequality (2.4) holds even if we take $N_A(T)$ to be the number of elements of $A$ in
{\em any} interval of length $2T$.
In fact, one can prove that $N_A(T)$ cannot be $0$ if $T$ is sufficiently large, depending on $\Omega$,
without taking the limit in inequality (2.4) and without talking about asymptotic density.
}
This uniformity of course implies that the maximum gap of $A$ is bounded by a quantity that depends
on $f$ only.

Since $\sum_{\lambda\in\Lambda}\Abs{\chi_\Omega}^2(x-\lambda)=1$ is a tiling and
$0 \le \Abs{\chi_\Omega}^2 \in L^1(\RR)$ we deduce that $\Lambda$ has gaps bounded above
by a function of $\Omega$ alone.
\end{proof}

Let now
$$
Z = \Set{\xi\in\RR:\ \ft{\chi_\Omega}(\xi) = 0}
$$
and define the finite set (as $Z$ is discrete)
\beql{sigma}
\Sigma = Z \cap (0,\Delta] = \Set{s_1, s_2, \ldots, s_k},
\eeq
where $\Delta$ is the quantity given by Corollary \ref{lm:bounded-gaps}.

We now view the set $\Sigma$ as a finite set of symbols (alphabet) and consider
the set $\Sigma^\ZZ$ of all bidirectional sequences of elements of $\Sigma$ equipped
with the product topology. A sequence $x^n$ of elements of $\Sigma^\ZZ$ converges to $x \in \Sigma^\ZZ$
if for all $k=1,2,\ldots$ the double sequences $x^n$ and $x$ agree in the window $[-k, k]$ for large enough
$n$.
More precisely, for all $k=1,2,\ldots$ there is $n_0$ such that for $n\ge n_0$
we have
$$
x^n_j = x_j,\ \ \mbox{for $-k \le j \le k$}.
$$
$\Sigma^\ZZ$ is a metrizable compact space so that each sequence $x^n \in \Sigma^\ZZ$ has a convergent
subsequence. This is just another way of phrasing a diagonal argument that is somewhat more
convenient to use. The proof below may of course be phrased avoiding topological notions altogether
and replacing the convergence of each subsequence with a diagonal argument.

The space $\Sigma^\ZZ$ is the natural space in which to view a spectrum $\Lambda$ of $\Omega$,
as the set $\Lambda$ is locally of {\em finite complexity}: because of 
\eqref{zeros-condition} the difference of any two successive elements of $\Lambda$ can
be only be an element of $\Sigma$.
By demanding, as we may, that $0$ is always in $\Lambda$ we can therefore represent
any set $\Lambda$ with the sequence of its successive differences.
More precisely, we map any set $\Lambda \subseteq \RR$ whose successive differences are in $\Sigma$
and which contains $0$
$$
\Lambda = \Set{\cdots < -\lambda_2 < -\lambda_{-1} < \lambda_0 = 0 < \lambda_1 < \lambda_2 < \cdots}
$$
to the element $\Lambda \in \Sigma^\ZZ$ given by
$$
\Lambda_n = \lambda_{n+1}-\lambda_n,\ \ \ (n \in \ZZ).
$$
This correspondence is a bijection and we will use one or the other form of the set $\Lambda$ as it
suits us.

\subsection{Symbolic sequences determined by their values in a half-line}
\label{sec:windows}

Suppose $X \subseteq \Sigma^\ZZ$. We say that $X$ is determined by left half-lines if knowing
an element of $X$ to the left of any index $n$ suffices to determine the element in the remaining
positions to the right of $n$, i.e.\ if for any $x, y \in X$ and $n \in \ZZ$ we have
$$
(x_i = y_i \mbox{ for $i\le n$}) \Longrightarrow (x_i = y_i \mbox{ for all $i \in \ZZ$}).
$$
Determination of $X$ by right half-lines is defined analogously.

We similarly say that $X$ is determined by any window of size $w$ (a positive integer)
if for any $x \in X$ and any $n\in\ZZ$
knowing $x_i$ for $i=n, n+1, \ldots, n+w-1$ completely determines $x$.

\begin{theorem}\label{th:diag}
Suppose $X \subseteq \Sigma^\ZZ$ is a closed, shift-invariant set which is determined by left half-lines
and by right half-lines.
Then there is a finite number $w$ such that $X$ is determined by windows of size $w$.
\end{theorem}

\begin{proof}
It is enough to show that there is a finite window size $w$ such that whenever two elements of $X$
agree on a window of size $w$ then they necessarily agree at the first index to the right of that window.
For in that case they necessarily agree at the entire right half-line to the right of the window and are
by assumption equal elements of $X$.

Assume this is not true. Then there are elements $x^n, y^n$ of $X$, $n=1,2,\ldots$,
which agree at some window of width $n$
but disagree at the first location to the right of that window. Using the shift invariance of $X$
we may assume that
$$
x^n_{-n}=y^n_{-n},\ x^n_{-n+1} = y^n_{-n+1},\ \ldots,\ x^n_{-1} = y^n_{-1}\ \ \&\ \ x^n_0 \neq y^n_0.
$$

By the compactness of the space there are $x, y \in X$ and a subsequence of $n$'s such that
$x^n \to x$ and $y^n \to y$.
By the meaning of convergence in the space $\Sigma^\ZZ$ we have that the sequences $x$ and $y$
agree for all negative indices and disagree at $0$.
This contradicts the assumption that $X$ is determined by left half-lines.
\end{proof}

\begin{theorem}\label{th:periodic-elements}
If $X \subseteq \Sigma^\ZZ$ is shift-invariant and is determined by windows of size $w$ then
all elements of $X$ are periodic, and the period can be chosen to be at most $\Abs{\Sigma}^w$.
\end{theorem}
\begin{proof}
Fix $x \in X$. Since there are at most $\Abs{\Sigma}^w$ different window-contents of length $w$, it
follows that there are two indices $i, j \in \Set{0,1,\ldots,\Abs{\Sigma}^w}$, $i<j$, such that
$$
x_i = x_j, x_{i+1} = x_{j+1}, \ldots, x_{i+w-1} = x_{j+w-1}.
$$
Writing $Tx$ for the left shift of $x \in X$ (i.e. $(Tx)_n = x_{n+1}$)
we have that $x$ and $T^{j-i}x$ agree at
the window $i, i+1, \ldots, i+w-1$. By assumption then $x = T^{j-i}x$, which is another way
of saying that the sequence $x$ has period $j-i \le \Abs{\Sigma}^w$.
\end{proof}

\subsection{Symbolic sequences with spectral gaps}
\label{sec:spectral-gaps}

Suppose $\Lambda \subseteq \RR$ is a spectrum of the bounded set $\Omega \subseteq \RR$ of measure 1.
Write $\delta_\Lambda = \sum_{\lambda\in\Lambda}\delta_{\lambda}$,
where $\delta_\lambda$ is a unit point mass at point $\lambda$.
It is well known
(see, for instance, \cite{kolountzakis2004milano})
that the Fourier Transform of the tempered distribution $\delta_\Lambda$ is supported
by 0 plus the zeros of the function
$$
\left(\Abs{\ft{\chi_\Omega}}^2\right)^\wedge = \chi_\Omega*\chi_{-\Omega}
$$
which is a continuous function with value 1 at the origin.
Therefore there is an interval $(0, a)$, with $a = a(\Omega) > 0$, such that
$\delta_\Lambda$ has a spectral gap:
\beql{spectral-gap}
\supp \ft{\delta_\Lambda} \cap (0, a) = \emptyset.
\eeq

With $\Sigma = \Sigma(\Omega)$ defined by \eqref{sigma}
let $X \subseteq \Sigma^\ZZ$ consist of all sequences which correspond to sets $\Lambda$
with gaps from $\Sigma$ such that \eqref{spectral-gap} holds.
The set $X$ is obviously shift-invariant
as shifting a sequence in $X$ corresponds to translation of the set $\Lambda$ and translation
will not affect the support of $\ft{\delta_\Lambda}$.
\begin{lemma}\label{lm:closed}
The set $X$ is closed in $\Sigma^\ZZ$.
\end{lemma}
\begin{proof}
Suppose $\Lambda^n \in X$ and $\Lambda^n \to \Lambda\in\Sigma^\ZZ$ and that $\phi \in C^\infty(0, a)$.
It is enough to show that $\ft{\delta_\Lambda}(\phi) = 0$ as this is what it means for $\ft{\delta_\Lambda}$
to have no support in $(0, a)$ and therefore $\Lambda \in X$.
By definition of the Fourier Transform
$$
\ft{\delta_\Lambda}(\phi) = \delta_\Lambda(\ft{\phi}) = \sum_{\lambda\in\Lambda} \ft{\phi}(\lambda)
 =^* \lim_{n\to\infty} \sum_{\lambda\in\Lambda^n} \ft{\phi}(\lambda) = \lim_{n\to\infty} \delta_{\Lambda^n}(\ft{\phi}) = \lim_{n\to\infty} \ft{\delta_{\Lambda^n}}(\phi) = 0.
$$
The justification for the equality $=^*$ above is very easy given (a) the rapid decay of $\ft{\phi}$, and, (b)
the fact that all $\Lambda^n$ have the same positive minimum gap.

Indeed, due to (a) and (b) we can find for any $\epsilon>0$ an $R>0$ such that
$$
\Abs{\sum_{\lambda\in L\ \&\ \Abs{\lambda}>R} \ft{\phi}(\lambda)} < \epsilon
  \ \ \mbox{for $L=\Lambda$ or $L=\Lambda^n$},
$$
and also an $n_0$ such that $\Lambda^n \cap [-R, R] = \Lambda \cap [-R, R]$ for $n\ge n_0$.
It follows that for $n\ge n_0$ we have
$$
\Abs{\ft{\delta_\Lambda}(\phi)} =
\Abs{\ft{\delta_\Lambda}(\phi) - \ft{\delta_{\Lambda^n}}(\phi)} =
\Abs{\sum_{\lambda\in \Lambda\ \&\ \Abs{\lambda}>R} \ft{\phi}(\lambda) - \sum_{\lambda\in \Lambda^n\ \&\ \Abs{\lambda}>R} \ft{\phi}(\lambda)} \le 2\epsilon.
$$
This implies that $\ft{\delta_\Lambda}(\phi)=0$ as we had to show.
\end{proof}

\begin{theorem}\label{th:by-half-lines}
The sequences in $X$ are determined by both left half-lines and right half-lines.
\end{theorem}
\begin{proof}
Suppose that $X$ is not determined by left half-lines (the argument is similar for right half-lines).
Then there are distinct $\Lambda^1, \Lambda^2 \in X$ such that
$\Lambda^1_i = \Lambda^2_i$ for all negative integers $i$.
Both $\delta_{\Lambda^1}$ and $\delta_{\Lambda^2}$ have a spectral gap at $(0,a)$ and therefore
so does their difference
$$
\mu = \delta_{\Lambda^1} - \delta_{\Lambda^2}.
$$
Notice that $\mu$ is supported in the half-line $[0, +\infty)$.
Suppose $\psi \in C^\infty(-a/10, a/10)$. It follows from the rapid decay of $\ft{\psi}$
that the measure
$$
\nu = \ft{\psi}\cdot \mu
$$
is totally bounded and
still has a spectral gap at the interval $(a/10, 9a/10)$.
But the measure $\nu$ is also supported in the half-line $[0, +\infty)$
and by the F. and M. Riesz Theorem \cite{havin1994uncertainty}
its Fourier Transform is mutually absolutely continuous with
respect to the Lebesgue measure on the line. But this is incompatible with the vanishing
of $\ft{\nu}$ in some interval. Therefore $\nu$ must be identically $0$ and,
since $\psi \in C^\infty(-a/10, a/10)$, is otherwise arbitrary, it follows that $\mu\equiv 0 $, or
$\Lambda^1 = \Lambda^2$, a contradiction. It follows that $X$ is indeed determined by left half-lines.
\end{proof}

\subsection{Conclusion of the argument}
\label{sec:conclusion}

By Lemma \ref{lm:closed} and Theorem \ref{th:by-half-lines} the set $X$ defined above,
right after \eqref{spectral-gap}, given $\Omega$
is a closed shift-invariant subset of $\Sigma^\ZZ$ and its elements are determined by half-lines.
By Theorem \ref{th:diag} there exists a finite number $w$
such that the elements of $X$ are determined by their values at any window of width
$w$. By Theorem \ref{th:periodic-elements} all elements of $X$
are therefore periodic sequences. Since all spectra of $\Omega$ can also
be viewed as elements of $X$,
the periodicity of any spectrum of $\Omega$ follows from the periodicity of the sequence
of its successive differences.

The fact that any period of $\Lambda$ is a positive integer is a consequence of
the fact that $\Lambda$ has density 1: if $T$ is a period of $\Lambda$ this implies that
there are exactly $T$ elements of $\Lambda$ in each interval $[x, x+T)$ hence $T$
is an integer.

\ 

\noindent{\bf Acknowledgment:} We are grateful to Dorin Dutkay for pointing out an error
in our initial proof of Lemma \ref{lm:bounded-gaps}.

\bibliographystyle{abbrv}
\bibliography{spectral-sets}

\end{document}